\documentclass{elsarticle}

\usepackage{here}
\usepackage{setspace}
\usepackage{array}
\usepackage{arydshln}
\usepackage{graphicx} \graphicspath{{pics/}}
\usepackage{natbib}
\usepackage[usenames,dvipsnames]{color}
\usepackage[utf8]{inputenc}
\usepackage{amsmath}
\usepackage{amsthm}
\usepackage{amsfonts}
\usepackage{amssymb}

\usepackage{tikz}
\usepackage{pgfplots}
\usepackage{mathdots}

\pgfplotsset{articleplot/.style={
    axis x line=bottom,
    axis y line=left,
    x axis line style={-}
  }}

\newcommand{\Nb}{\mathbb{N}}
\newcommand{\drm}{\mathrm{d}}
\newtheorem{prop}{Proposition}
\newtheorem{coro}{Corollary}

\newtheorem{lemma}{Lemma}

\newtheorem{theorem}{Theorem}
\newcommand{\old}[1]{{}}

\begin{document}

\title{On the Number of Non-zero Elements of Joint Degree Vectors}

\author[sc]{\'Eva Czabarka}\ead{czabarka@math.sc.edu}
\author[luh,mpi]{Johannes Rauh}\ead{jrauh@mis.mpg.de}
\author[uk]{Kayvan Sadeghi}\ead{k.sadeghi@statslab.cam.ac.uk}
\author[gv]{Taylor Short}\ead{shorttay@gvsu.edu}
\author[sc]{L\'aszl\'o Sz\'ekely}\ead{szekely@math.sc.edu}
\address[sc]{University of South Carolina, Columbia}
\address[luh]{Leibniz University of Hannover}
\address[mpi]{Max Planck Institute for Mathematics in the Sciences}
\address[uk]{University of Cambridge}
\address[gv]{Grand Valley State University}


\begin{abstract}
Joint degree vectors give the number of edges between vertices of degree $i$ and degree $j$ for $1\le i\le j\le n-1$
  in an $n$-vertex graph.  We find lower and upper bounds for the maximum number of nonzero elements in a joint degree
  vector as a function of~$n$.  This provides an upper bound on the number of estimable parameters in the exponential random graph model
 with bidegree-distribution as its sufficient statistics.
\end{abstract}

\maketitle

\section{Introduction}

\emph{Degree sequences} and \emph{degree distributions} have been subjects of study in graph theory and many other fields in the past decades.
In particular, in social network analysis, they have been shown to possess a great expressive power in representing and statistically modeling networks; see, e.g., \citet{new03} and \citet{han07}. Generally in this context, models are in exponential family form \citep{bar78,bro86}, known as \emph{exponential random graph models} (ERGMs) \citep{fra91,was96}. When the sufficient statistic, i.e.\ the only information that the ERGM gathers from an observed network, is the degree sequence of a network, the corresponding
ERGM is known as the \emph{beta-model}, properties of which have been extensively studied in the recent literature; see \citet{bli10},  \citet{cha11}, and \citet{rin13}.
Degree distributions have also been used as sufficient statistics; see \citet{sadr14}.

The \emph{bidegree distribution} generalizes the degree distribution and collects the relative frequencies of the degree
combinations that appear at neighbouring vertices.  
The non-normalized version of the bidegree distribution is called the \emph{joint degree vector} (JDV, sometimes also called \emph{joint degree matrix}) \citep{pat76,ama08,sta12},
i.e.\ the elements of the JDV represent the exact counts
of edges between pairs of vertices of specified degree.  Conditions for a given vector to be the JDV of a graph were provided in
\citet{pat76},  \citet{sta12}, and
 \citet{cza15}.
An ERGM with bidegree distribution as sufficient statistics has been formalized in \citet{sadr14}. 

Bidegree distributions are network statistics that belong to the more general class of joint degree distributions that count degree
combinations of connected sets of vertices of given size.
In the computer science literature, the family of graphs with a given joint degree distribution is called \emph{$d$K-graphs}, where $d$ indicates the number of vertices of the concerned subgraphs~\citep{mah06}. The class of $d$K-graphs was originally formulated as a means to capture increasingly refined properties of
networks in a hierarchical manner based on higher order interactions among vertex degrees (see, e.g., \citet{dim09}).



One important statistical problem when working with ERGMs (or other exponential families) is the question whether the \emph{maximum likelihood estimate} (MLE) exists for a given set of observations (in network theory, the observations most often consist just of a single observed network).
When the MLE does not exist, one or more of the model parameters cannot be estimated.
As is well-known, the information when the MLE exists can be obtained from a
facet description of the so called \emph{model polytope} \citep{bar78}.  These facets correspond to linear inequalities
that hold among the different components of the sufficient statistics.  When such a description is known, it is also
easy to understand which parameters can be estimated~\citep{WangRauhMassam16:Approximating_faces}.

Even though a complete description of the model polytope is hard to compute in general, it is often possible to obtain a subset of the valid inequalities.  Such partial information about the model polytope gives partial information about MLE existence and parameter behaviour~\citep{WangRauhMassam16:Approximating_faces}.  For example, the bidegree counts (and frequencies) are always non-negative.
\citet{sadr14} exploited these facts to show that the MLE never exists for a single observed network in the case of bidegree distribution as sufficient statistics.  More importantly, parameters corresponding to zeros on the bidegree vector of the observed network are not estimable.

These results motivates us to find the maximum possible number of non-zero elements on the bidegree vector of a graph.
This maximum number also tells us about the maximum number of estimable parameters.  In this paper, we prove that,
asymptotically for large~$n$, the maximum number of non-zero elements lies between $0.5\binom{n}{2}$ and
$\frac{13}{24}\binom{n}{2}\approx0.541\bar{6}\binom{n}{2}$.  Thus, roughly half of the components are zero, and so, at
most half of the parameters are estimable from a single observation.


In the next section, we provide basic graph theoretical as well as statistical definitions and preliminary results.  In Section~\ref{sec:3}, we provide a lower bound for the maximum possible number of non-zero elements of a JDV by constructing a family of graphs that reaches this bound. In Section~\ref{sec:40}, we exploit conditions from \citet{cza15} and use two different approaches to obtain upper bounds for this desired value. The first upper bound is presented in Theorem~\ref{thm:1} and the second bound in Theorem~\ref{mainthm}. As shall be seen, the numerical values for the two bounds are very close.

\section{Definitions and preliminary results}\label{sec:2}

\subsection{Joint degree vectors}
In this paper we only consider simple graphs without isolated vertices. Let $G=(V,E)$ be such an $n$-vertex graph and for $1\le i \le n-1$ let $V_i$ be the set of vertices of degree $i$. The \emph{joint degree vector} (JDV) of $G$ is the vector ${\bf s(G)} =(j_{11}(G),j_{12}(G),\ldots,j_{n-1,n-1}(G))$ of length $\binom{n}{2}$ with components defined by $j_{ik} = |\{ xy \in E(G) : x \in V_i, y \in V_k \}|$ for all $1\le i\le k\le n-1$.  For some vector ${\bf m}$, if there exists a graph $G$ with ${\bf s(G)} = {\bf m}$, then $\mathbf{m}$ is called a \emph{graphical JDV}. Note that the degree sequence of a graph is determined by its JDV in that
$$
|V_i|=\frac{1}{i}\left(\sum_{k=1}^i j_{ki}+\sum_{k=i}^{n-1} j_{ik}\right).
$$
The following characterization for a vector ${\bf m}$ with integer entries to be a graphical JDV is proved by \citet{pat76}, \citet{sta12}, and \citet{cza15}. As it provides simple necesssary and sufficient conditions for a vector to be realized as a graphical JDV, we call the result an Erd\"{o}s-Gallai type theorem.

\begin{prop}(Erd\H{o}s-Gallai type theorem for a JDV) \label{EGT}
An integer vector ${\bf m}=(m_{11},m_{12}, \ldots, m_{n-1,n-1})$ of size $\binom{n}{2}$ is a graphical JDV if and only if the following holds:
\begin{enumerate}
\item[(i)] for all $i$: $\displaystyle n_i := \frac{1}{i} \left( \sum_{k=1}^{i} m_{ik}+\sum_{k=i}^{n-1}m_{ik} \right)$ is an integer,
\item[(ii)] for all $i$: $\displaystyle m_{ii} \le \binom{n_i}{2}$,
\item[(iii)] for all $i < k$: $\displaystyle m_{ik} \le n_in_k$.
\end{enumerate}
Moreover, $n_i$ gives the number of vertices of degree $i$ in the graph $G$.
\end{prop}

 \subsection{Exponential random graph models}

 An \emph{exponential random graph model} (ERGM) is a family of random graphs, parametrized by finitely many parameters $\theta_{i}$, $i\in I$.  All random graphs have the same (finite) set of vertices, denoted by~$V$.  Under this model, the probability of observing a network $G$ with vertex set $V$ can be written as
 \begin{equation}\label{eq:ergm}
 P(G)=\exp\{\sum_{i\in I}t_i(G)\theta_i-\psi(\theta)\},
 \end{equation}
 where $t_i(G)$ are \emph{canonical sufficient statistics}, which capture some important feature of $G$, and $\psi(\theta)$ is the \emph{normalizing constant}, which ensures that probabilities add to $1$ when summing over all possible networks.

 The model is in exponential family form. Hence, the likelihood function $l(\theta)= P(G_1,\dots,G_m)$, for generic observed networks $G_1,\dots,G_m$, is concave and, therefore, has a unique maximum if it exists. Existence of this maximum can be described geometrically:

 Suppose that the networks $G_1,\dots,G_m$ were observed. The \emph{average observed sufficient statistic} $\bar{t}$ is
 $\bar{t}_i=\frac{1}{m}\sum_{j=1}^m t_i(G_j)$, $1\leq i\leq d$. We also define \emph{the model polytope} to be the convex hull of all the points in a $d$-dimensional space that correspond to the sufficient statistics of all graphs with $n$ vertices. We then have the following
 result \citep{bar78,bro86}:
 \begin{prop}
 For an ERGM, the MLE exists if and only if the average observed sufficient statistic $\bar{t}$ lies in the (relative) interior of the model polytope.
 \end{prop}
In network analysis, there is usually only one network $G$ observed, and therefore, the average 
 observed sufficient statistic is simply $t(G)$.

 In the so-called $2$K-model, the sufficient statistic $t(G)$ in (\ref{eq:ergm}) is the JDV~${\bf s(G)}$.  As shown in \citep{sadr14}, if $s_i(G)=0$, then $\theta_i$ is not estimable. It is also easy to observe that for every graph, there are always some elements of the bidegree vector that are zero. In the next sections, we investigate how many elements of the bidegree vector are always zero.

\section{Lower bound construction}\label{sec:3}
Let $H_n$ denote an $n$-vertex graph with vertex set $V(H_n) = \{v_1, v_2, \ldots, v_n\}$ and edge set $E(H_n) = \{v_iv_j : i+j > n \text{ and } i\neq j\}$. This graph, which is known as the \emph{half graph}, has degree sequence $n-1, n-2, \ldots, \left \lfloor \frac{n}{2} \right \rfloor, \left \lfloor \frac{n}{2} \right \rfloor,  \ldots, 2, 1$. Since a graph on $n$ vertices cannot contain both vertices with degrees $0$ and~$n-1$, the half graph attains the maximum number of distinct degrees.

For any graph $G$, let
\begin{equation*}
A(G) = \{ ik : j_{ik}\neq 0\}
\end{equation*}
be the set of non-zero components in the JDV of~$G$. Clearly, $|A(H_n)| = \frac{n^2}{4}$ if $n$ is even and $|A(H_n)| = \frac{n^2-1}{4}$ if $n$ is odd. Hence,
$$
\lim _{n\rightarrow \infty} \frac{|A(H_n)|}{\binom{n}{2}} = \frac{1}{2},
$$
so about half the elements of the JDV of the half graph are non-zero.

The half graphs are not optimal, and there are constructions which achieve a higher number of non-zero elements in the JDV.  Consider the graph $H_n$ with $n\geq 7$ odd.  If one connects the degree $1$ vertex to one of the vertices with degree $(n-1)/2$, the JDV element $j_{1,n-1}$ becomes $0$, but the elements $j_{2,(n+1)/2}$ and $j_{(n+1)/2,(n+1)/2}$ become nonzero, so the new graph has one more nonzero elements in its JDV.  We found even better such constructions, but all of these only improve $|A(H_n)|$ by a term that is linear in~$n$.

\section{Two upper bounds}\label{sec:40}

In this section, we provide two upper bounds that provide numerically very close upper bounds, but use entirely different methods.
Although we tried, we were unable to combine these two proof techniques. We think that it is instructive to show both of them.

\subsection{Continuous optimization}\label{sec:4}

The following identity is a simple consequence of Proposition~\ref{EGT} and is due to \citet{sadr14}:
\begin{prop}\label{prop:8}
For any graph $G$,
\begin{equation}\label{eq:13}
  \sum_{(k_1,k_{2})\in A(G)} \frac{k_1+k_2}{k_1k_2}j_{k_1k_2}(G)=n-n_0(G),
\end{equation}
where $n_0(G)$ is the number of isolated vertices in $G$.
\end{prop}
To see this, by Proposition~\ref{EGT}(i) we have
\begin{align*}
n-n_0(G)&=\sum_{i=1}^{n-1} n_i(G)=\sum_{i=1}^{n-1}\frac{1}{i}\left(\sum_{k=1}^i j_{k i}(G)+\sum_{k=i}^{n-1}j_{ik}(G) \right)\\
&=\sum_{(k_1,k_2)\in A(G)}\left(\frac{1}{k_1}+\frac{1}{k_2}\right)j_{k_1k_2}(G)=\sum_{(k_1,k_2)\in A(G)}\frac{k_1+k_2}{k_1k_2}j_{k_1k_2}(G)
\end{align*}

Next, we show that we can assume that $n_{0}(G) = 0$ without loss of generality.
Consider a graph $G$ with $n_0(G)>0$. If $n_0(G)=1$, then let $v\in V(G)$ be a largest degree vertex in $G$ and $x$ be the isolated vertex, and let $G'$ be the graph obtained by adding the edge $xv$ to
the graph $G$. If $n_0(G)>1$ then let $G'$ be the graph obtained by adding the edges between the isolated vertices of~$G$.  In both cases $G'$ is a graph on the same number of vertices as $G$ that
has at least as many nonzero entries in its JDV as $G$ does.
Thus, there are graphs without isolated vertices that have the maximum number of nonzero entries in their JDV.

\begin{coro}
  For any graph $G$,
    \begin{equation*}
    \sum_{(k_{1},k_{2})\in A(G)}\frac{k_{1}+k_{2}}{k_{1}k_{2}} \le \sum_{(k_1,k_{2})\in A(G)}\frac{k_1+k_2}{k_1k_2}j_{k_1k_2}(G) \le n. 
  \end{equation*}
\end{coro}

The original problem of finding the maximum possible number of non-zero elements of a JDV for a fixed number of vertices
can be formulated as the following optimization problem:
\begin{itemize}
\item \emph{ Maximize $|A(G)|$ among all graphs~$G$ with $n$ vertices.}
\end{itemize}
Using the corollary, we relax this optimization problem and study the following problem, which we will refer to as the
\emph{discrete relaxation} (as ultimately we will solve its continuous version):
\begin{itemize}
\item
  \emph{Maximize the cardinality $|A|$ among all subsets $A \subseteq P_{n}:= \{(i,j)\in\Nb^{2} : 1\le i\le j\le n-1\}$
    under the constraint $\sum_{(k_{1},k_{2})\in A} \frac{k_{1}+k_{2}}{k_{1}k_{2}} \le n$.}
\end{itemize}
By the above corollary,  for any~$n$, the cardinality of a subset that solves the discrete relaxation is an upper bound for the original optimization problem.

The discrete relaxation can be solved on a computer as follows:
First, compute all values $(k_1+k_2)/(k_1k_2)$ on~$P_{n}$.  Second, order the values.  Third, start adding them up as
long as the sum does not exceed~$n$.  Finally, count the number of elements that have been added.
Let $\alpha_{n}$ be the cardinality of a solution $A$ of the discrete relaxation divided by~$\binom{n}{2}$, the cardinality of~$P_{n}$.
The values of $\alpha_{n}$ are plotted in Figure~\ref{fig:G}.
As a function of~$n$, the optimum $\alpha_{n}$ decreases roughly (though not strictly) and reaches values below~$0.56$
for large~$n$.
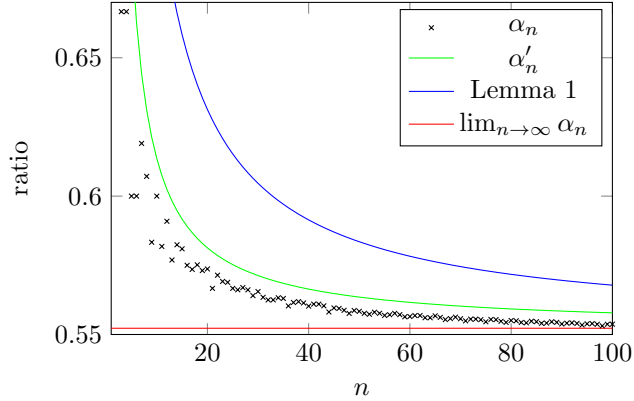
\begin{figure}
  \centering
  \begin{tikzpicture}
    \begin{axis}[width=0.68\linewidth, height=6cm,
      xlabel={$n$}, ylabel={ratio},
      xmin=1, xmax=100, ymin=0.55, ymax=0.67,
      legend entries={$\alpha_{n}$,$\alpha_{n}'$,{Lemma~\ref{lem:discrete-vs-continuous}},$\lim_{n\to\infty}\alpha_{n}$}
      ]
      \addplot[black,only marks,mark size=1.3pt,mark=x] table[x index=0, y index=1] {discrete-rel.dat};
             \addplot[green,domain=1:100,samples=100]{0.55225694*x/(x-1)};
      \addplot[blue,domain=1:100,samples=100]{0.55225694*x/(x-1)+1/x};
      \addplot[red,domain=1:100,samples=2]{0.55225694};
     \end{axis}
  \end{tikzpicture}
  \caption{The solution of the discrete and continuous relaxation.  The blue curve plots the upper bound on~$\alpha_{n}$
    from Lemma~\ref{lem:discrete-vs-continuous}.  The red curve is the limit for large~$n$ of $\alpha_{n}$
    and~$\alpha_{n}'$.}
  \label{fig:G}
\end{figure}

The limit for $n\to\infty$ can be computed by approximating the discrete relaxation by the following optimization problem, which we call the \emph{continuous relaxation}:
\begin{itemize}
\item \emph{Maximize $\frac{\mu(A')}{(n-1)^{2}}$ (where $\mu$ denotes the Lebesque measure) among all subsets $A'\subseteq[1,n]\times[1,n]$ that
are symmetric to the line $y=x$ and
  satisfy\\ \mbox{$\iint_{A'}\frac{1}{x}\drm x\drm y\le n$}.}
\end{itemize}
Let $\alpha_{n}'$ be the maximum of the continuous relaxation.
\begin{lemma}
  \label{lem:discrete-vs-continuous}
  $\alpha_{n} \le \frac{n-1}{n}\alpha_{n}' + \frac{1}{n}$.
\end{lemma}
\begin{proof}
  To each $(i,j)\in P_{n}$ associate the two squares $A_{i,j}:=[i,i+1)\times[j,j+1)$ and
  $A_{j,i}:=[j,j+1)\times[i,i+1)$.  For $A\subseteq P_{n}$ let $A''=\bigcup_{(i,j)\in A}A_{i,j}$ and
  $A'=A''\cup\bigcup_{(i,j)\in A}A_{j,i}$.  Then
  \begin{multline*}
    \sum_{(i,j)\in A}\frac{i+j}{i\cdot j}
    \ge\sum_{(i,j)\in A}\iint_{A_{i,j}}\frac{x+y}{x\cdot y}\drm x\drm y
    = \iint_{A''}\frac{x+y}{x\cdot y}\drm x\drm y \\
    \ge \frac12 \iint_{A'}\frac{x+y}{x\cdot y}\drm x\drm y
    = \iint_{A'}\frac{1}{x}\drm x\drm y.
  \end{multline*}
  Here, the first inequality follows from the fact that the maximum of $\frac{x+y}{xy}=\frac1x+\frac1y$ over~$A_{i,j}$
  is at $(x,y)=(i,j)$.  The second inequality follows by not double-counting the set $A_{d}:=\bigcup_{(i,i)\in
    A}A_{i,i}$ corresponding to the diagonal elements of~$A$.  The last equality follows since $\frac{x+y}{x\cdot
    y}=\frac1x+\frac1y$ and since $A'$ is symmetric about $y=x$.  Therefore, if $A$ is feasible for the discrete
  relaxation, then $A'$ is feasible solution for the continuous relaxation.
  Now,
  \begin{equation*}
    |A|=\mu(A'')=\frac{1}{2}(\mu(A') + \mu(A_{d}))\le\frac{\mu(A')}{2} + \frac{n-1}{2},
  \end{equation*}
  and so
  \begin{equation*}
    \alpha_{n} \le \frac{(n-1)^{2}}{2\binom n2}\alpha'_{n} + \frac{n-1}{2\binom n2}
    =\frac{n-1}{n}\alpha'_{n}+\frac{1}{n}. \qedhere
  \end{equation*}
\end{proof}
\begin{coro}
  $\limsup_{n\to\infty} \alpha_{n}\le\limsup_{n\to\infty}\alpha_{n}'$.
\end{coro}
It is not difficult to see that, actually, $\lim_{n\to\infty} \alpha_{n}=\lim_{n\to\infty}\alpha_{n}'$.
Figure~\ref{fig:G} shows that the upper bound from Lemma~\ref{lem:discrete-vs-continuous} is not very tight for finite~$n$.

Next, we solve the continuous relaxation.  The idea is the following: As the set $A'$ it is advantageous to
choose a sublevel set of the function~$\frac{x+y}{x\cdot y}$.
For $c>0$ let
\begin{equation*}
  A_{c} := \Big\{ (x,y)\in[1,n]^{2} : \frac{x+y}{x\cdot y}\le c \Big\}.
\end{equation*}
Let
\begin{align*}
  y_{c}(x) &= \frac{1}{c - \frac1x} = \frac{x}{xc-1}, &
  x_{1}(c) &= \frac{1}{c - \frac1n} = \frac{n}{nc-1}.
\end{align*}
\begin{lemma}
  \label{lem:Ac}
  $A_{c} = \Big\{ (x,y)\in[1,n]^{2}: x_{1}(c) \le x \le n, \; y_{c}(x) \le y\le n\Big\}$.  In particular,
  $A_{c}\neq\emptyset$ if and only if $nc\ge 2$.
\end{lemma}
\begin{proof}
  If $x<x_{1}(c)$ and $1\le y\le n$, then $\frac1x+\frac1y> c - \frac1n + \frac1n = c$.  If $x_{1}(c)\le x\le n$ and
  $1\le y< y_{c}(x)$, then $\frac1x+\frac1y> c - \frac1x + \frac1x = c$.  For the second statement observe that
  $x_{1}(c)\le n$ if and only if~$nc\ge 2$.  Similarly, $y_{c}(x)\le n$ if and only if $x\ge x_{1}(c)$.
\end{proof}
\begin{lemma}
  Assume that $c$ is such that $x_{1}(c)\ge 1$.  Then $y_{c}(x)\ge 1$ for all~$x\in[1,n]$.
\end{lemma}
\begin{proof}
  $y_{c}(x)$ decreases monotonically with~$x$.  Therefore, $y_{c}(x)\ge y_{c}(n)=x_{1}(c)$ for all~$x\in[1,n]$.
\end{proof}
\begin{lemma}
  Let $n\ge 3$.
  The set $A_{c}$ is feasible for the continuous relaxation if and only if
  \begin{equation}
    \label{eq:crit-c-ineq}
    (nc-2)\log(nc-1) \le nc
  \end{equation}
\end{lemma}
\begin{proof}
  Assume that $c$ is such that $x_{1}(c)\ge 1$.  Then
  \begin{multline*}
    \iint_{A_{c}}\frac{1}{x}\drm x\drm y = \int_{x_{1}(c)}^{n}\drm x\int_{y_{c}(x)}^{n} \drm y \frac1x
    = \int_{x_{1}(c)}^{n}\drm x\frac{n-y_{c}(x)}{x} \\
    = \int_{x_{1}(c)}^{n}\drm x\left(\frac{n}{x}-\frac{1}{xc-1}\right)
    = n\log\frac{n}{x_{1}(c)} - \frac{1}{c}\log\frac{nc-1}{cx_{1}(c)-1}.
  \end{multline*}
  Now,
  \begin{equation*}
    c x_{1}(c)-1 = \frac{cn - nc + 1}{nc - 1} = \frac{1}{nc-1},
  \end{equation*}
  and so
  \begin{equation*}
    \iint_{A_{c}}\frac{1}{x}\drm x\drm y
    = n\log(nc-1) - \frac1c \log(nc-1)^{2} = (n-\frac2c)\log(nc-1).
  \end{equation*}
  Hence, $A_{c}$ is feasible if and only if
  \begin{equation*}
    (nc-2)\log(nc-1) \le nc.
  \end{equation*}

  Now suppose that $n>e$.  If $c$ satisfies~\eqref{eq:crit-c-ineq}, then
  \begin{equation*}
    x_{1}(c) \ge \frac{n}{\exp(nc/(nc-2))} > \frac ne > 1.
  \end{equation*}
  Thus,
  the above calculation is valid and shows that $A_{c}$ is feasible.  On the other hand, if $n>e$ and if $c$
  violates~\eqref{eq:crit-c-ineq}, then $A_{c}$ is not feasible.
\end{proof}
To find the solution of the continuous relaxation, we need to find the value of $c$ that solves~\eqref{eq:crit-c-ineq}
with equality.  Consider
the equation
\begin{equation*}
\log(\beta-1)=\frac{\beta}{\beta-2}.
\end{equation*}
Both the left and the right hand side
change sign at $\beta=2$.  For $\beta>2$, both sides are positive, and for $\beta<2$ they are negative.
By Lemma~\ref{lem:Ac}, we are looking for a solution larger than~2.  For $\beta>2$, the right hand side is decreasing,
while the left hand side is increasing.  It follows that there is a unique solution~$\beta_{0}>2$.  Numerically,
$\beta_{0}\approx 5.68050$.  Thus, $A_{c}$ is feasible if and only if~$c\le\beta_{0}/n$, and
in order to maximize $|A_{c}|$, we have to choose $c=\beta_{0}/n$.
\begin{lemma}
  $x_{1}(\beta_{0}/n)>1$ for $n$ large enough.
  \end{lemma}
\begin{proof}
  $x_{1}(\beta_{0}/n) - 1
    = \frac{n - \beta_{0}+1}{\beta_{0}-1}>0$ for $n$ large enough.
\end{proof}
It remains to compute the maximum value of the continuous relaxation and to put everything together.
\begin{theorem}\label{thm:1}
  For any graph $G$ with $n$ vertices,
  \begin{equation*}
    \frac{|A(G)|}{\binom{n}{2}} \le \alpha_{n}' = \frac{n^{2}}{(n-1)^{2}}\frac{(\beta_{0}-2)^{2}-2}{\beta_{0}(\beta_{0}-2)}
    \approx \frac{n^{2}}{(n-1)^{2}}0.55225694,
  \end{equation*}
  where $A(G)$ is the set of non-zero elements in the JDV of~$G$.
\end{theorem}
\begin{proof}
  If $x_{1}(c)\ge 1$, then
\begin{equation*}
  |A_{c}|=\iint_{A}\drm x\drm y
  = \int_{x_{1}(c)}^{n}\drm x\int_{y_{c}(x)}^{n} \drm y
  = \int_{x_{1}(c)}^{n}\drm x(n-y_{c}(x)).
\end{equation*}
Now,
\begin{equation*}
  y_{c}(x) = \frac1c\left(\frac{x}{x-\frac1c}\right)
  = \frac1c\left(1 + \frac{1/c}{x-\frac1c}\right)
  = \frac1c\left(1 + \frac{1}{cx-1}\right),
\end{equation*}
and so
\begin{multline*}
  |A_{c}|
  = \int_{x_{1}(c)}^{n}\drm x(n-\frac1c - \frac{1/c}{cx-1})
  = (n-\frac1c)(n-x_{1}(c)) - \frac{1}{c^{2}}\log\frac{cn-1}{cx_{1}(c)-1} \\
  = n^{2}\frac{nc-1}{nc}\frac{nc-2}{nc-1} - \frac{2}{c^{2}}\log(nc-1).
\end{multline*}
Therefore,
\begin{equation*}
  \alpha_{n}' = \frac{|A_{\beta_{0}/n}|}{(n-1)^{2}}
  = \frac{n^{2}}{(n-1)^{2}}\left[\frac{\beta_{0}-2}{\beta_{0}} - \frac{2}{\beta_{0}^{2}}\frac{\beta_{0}}{\beta_{0}-2}\right]
  = \frac{n^{2}}{(n-1)^{2}}\frac{(\beta_{0}-2)^{2}-2}{\beta_{0}(\beta_{0}-2)}.
\end{equation*}
\end{proof}

\subsection{Second Bound}\label{sec:5}

Let $G=(V,E)$ be an $n$-vertex graph and let $A(G)$ be the set of non-zero elements in the JDV of $G$, as defined as in Section~$\ref{sec:4}$.  Denote by $n_i=|V_{i}|$ the number of vertices with degree $i$.  We call $i$ a single if $n_i=1$ and multiple if $n_i\geq 2$, noting that some $i$ are neither single nor multiple, as they just do not occur as degrees. As before, for $1\le i\le k\le n-1$, let $j_{ik}$ be the number of edges between the $i$th and $k$th degree classes and
$\chi_{ik}=1$ if $j_{ik}>0$, and 0 otherwise. It is easy to see that $|A(G)|=\sum_{i=1}^{n-1}\sum_{k=i}^{n-1}\chi_{ik}$. Now we set $D_i=\sum_{k=1}^i \chi_{ki}+\sum_{k=i+1}^{n-1}\chi_{ik}$ and $B(G)=\sum_{i=1}^{n-1} D_i$. Note that for $k\neq i$, $D_i$ counts $\chi _{ki} = \chi _{ik}$ twice but $\chi _{ii}$ is counted only once, so we get $|A(G)|\le \frac{B(G) + n-1}{2}$ and therefore
$$
\frac{|A(G)|}{\binom{n}{2}} \le \frac{B(G)+n-1}{2}\cdot\frac{2}{n(n-1)}=(1+o(1))\frac{B(G)}{n^2}.
$$
We use this to prove the following theorem:

\begin{theorem}
  \label{mainthm}
  For any graph $G$ with $n$ vertices,
  $$
  \frac{|A(G)|}{\binom{n}{2}} \le (1+o(1))\frac{13}{24} = (1 + o(1)) 0.541\overline{6},
  $$
  where $A(G)$ is the set of non-zero elements in the JDV of~$G$.
\end{theorem}
The proof of Theorem~\ref{mainthm} relies on a sequence of lemmas.
\begin{lemma} \label{lemmaD}
Let $m$ be the number of distinct vertex degrees of~$G$.  Then,
$$
\sum _{i=1}^{n-1} D_i \le \sum _{i:\, i \text{\textup{ single}}} \min(m, i) + \sqrt{m} \sqrt{\sum _{i:\, i \text{\textup{ multiple}}}\min(m,i)} \sqrt{\sum _{i:\, i \text{ \textup{multiple}}} n_i}.
$$
\end{lemma}

\begin{proof}
Observe that $D_i\leq m\leq mi$ and $D_i\leq in_i$, and hence
\begin{equation}
  \label{multiple}
  D_i\leq \min (m, in_i,mn_{i}) = \min (m, \min(m,i)\cdot n_i) \le \sqrt{m \cdot \min(m,i)\cdot n_i},
\end{equation}
since the minimum of two elements is less than their average. 
Note that if $i$ is single we have
\begin{equation} \label{single}
D_i \le \min (m,i).
\end{equation}
Employing $(\ref{multiple})$ and $(\ref{single})$ we get that
\begin{align}
\sum_{i=1}^{n-1} D_i &\leq \sum _{i:\, i \text{ single}} D_i + \sum _{i:\, i \text{ multiple}} D_i \notag \\
&\le \sum _{i:\, i \text{ single}} \min(m, i) + \sqrt{m} \sum _{i:\, i \text{ multiple}}\sqrt{\min(m,i)\cdot n_i} \notag \\
&\le \sum _{i:\, i \text{ single}} \min(m, i) + \sqrt{m} \sqrt{\sum _{i:\, i \text{ multiple}}\min(m,i)} \sqrt{\sum _{i:\, i \text{ multiple}} n_i}, \label{upper}
\end{align}
where the last inequality follows from Cauchy-Schwarz.
\end{proof}

We wish to upper bound the term from $(\ref{upper})$ over all graphs $G$. From our lower bound construction we know that $|A(G)| \ge (1-o(1))\frac12 n^2$. So we may assume that $m > n/\sqrt{2}$, else we would have $|A(G)| \le m^2 \le n^2/2$ and our estimation of $|A(G)|$ would be complete.

\begin{lemma} \label{sum}
$
\displaystyle
\sum _{i:n_i>0} \min(m, i) \le m(n-m-1) + \frac{n(2m-n+1)}{2}.
$
\end{lemma}

\begin{proof}
We wish to upper bound $\sum _{i:n_i>0} \min(m, i)$ over all graphs. So assume the $m$ highest possible degrees occur in our graph:
$$
n-1, n-2, \ldots, n-m+1, n-m.
$$
Our assumption $m > n/\sqrt{2}$ implies $m > n-m+1$, so the value of $m$ has to appear in the list of degrees above. There are $n-1-m$ terms strictly larger than $m$ in this list and each contributes $\min(m, i) = m$. The remaining terms  sum  up  exactly $\sum _{i=n-m}^m i$, and hence
\begin{align}
\sum _{i:n_i>0} \min(m, i) &\le m(n-m-1) + \sum _{i=n-m}^m i \notag \\
&= m(n-m-1) + \frac{n(2m-n+1)}{2}.\label{degreesum}
\end{align}
Now if the $m$ highest degrees do not occur in our graph, then some degree less than $n-m+1$ must occur which clearly gives something smaller than the term in $(\ref{degreesum})$.
\end{proof}

Recall from the beginning of the section that a degree $i$ is single if $n_i=1$, that is, there is only one vertex of degree $i$. Let $s$ be the number of degrees $i$ that are singles. Observe that $s \le m$ and $s+2(m-s) \le n$, implying that $s \le m \le \frac{n+s}{2}$.  Using $s+\sum_{i: i\ multiple}n_i=n$ and substituting
\begin{align*}
y &= \sum _{i:i \text{ single}} \min(m, i), &
z &= \sqrt{\sum _{i:i \text{ multiple}}\min(m,i)},
\end{align*}
we can write the term in $(\ref{upper})$ as
\begin{equation*}
g(y,z,s,m) = y + \sqrt{m} \cdot z \sqrt{n-s}.
\end{equation*}
We wish to maximize $g$ subject to the constraints
\begin{enumerate}
\item All variables are non-negative and $s\leq n$,
\item $s \le m \le \frac{n+s}{2}$, and
\item $y+z^2 \le m(n-m-1) + \frac{n(2m-n+1)}{2}$,
\end{enumerate}
where constraint 3 follows from Lemma $\ref{sum}$.

Note that $g(y,z,s,m) = O(n^2)$, so we wish to determine how large   the coefficient of $n^2$ in $g$ can be as $n\rightarrow \infty$. To do this we set $S = s/n$, $M = m/n$, $Y = \sum _{i:i \text{ single}} \min(m, i)/n^2$, and $Z = \sqrt{\sum _{i:i \text{ multiple}}\min(m,i)/n^2}$ and turn to the following numeric optimization problem: maximize $$f(Y,Z,S,M) = Y + \sqrt{M} \cdot Z \sqrt{1-S}$$ subject to the constraints
\begin{enumerate}
\item[(a)] All variables  are non-negative and  $S\leq 1$,
\item[(b)] $S \le M \le \frac{1+S}{2}$, and
\item[(c)] $Y+Z^2 \le M(1-M) + \frac{2M-1}{2}$.
\end{enumerate}
By routine arguments, it follows that $g(y,z,s,m) \le \left( 1 + o(1) \right)f(Y,Z,S,M)n^2$.
Also note that (b) implies that $M\le1$.

\begin{lemma} \label{lemmaf}
If constraints (a), (b), and (c) hold, then
$$
f(Y,Z,S,M) \le \frac{13}{24}.
$$
\end{lemma}

\begin{proof}
For fixed values of $S$, $M$ and $Z$, the function $f$ is monotone in $Y$.  Therefore, we have to choose $Y$ as large as possible, which, according to the last constraint, implies that $Y = -M^{2} + 2M - \frac12 - Z^{2}$. Also the right hand side of constraint (c) is non-negative if and only if $1-\sqrt{2}/2 \le M\le 1+\sqrt{2}/2$, so $1-\sqrt{2}/2\le M\le 1$.

For fixed values of $Z$ and $M$ the target function $f$ decreases with $S$.  Hence we need to choose $S$ as small as possible.  The constraints imply $S\ge\max\{0,2M-1\}$, so we consider the following two cases.

If $M\le 1/2$, then $S=0$.  In this case, we need to optimize
\begin{equation*}
  f(Z,M) = -M^{2} + 2M - \frac12 - Z^{2} + \sqrt{M}\cdot Z
\end{equation*}
subject to
\begin{enumerate}
\item $1 - \sqrt{2}/2\le M \le 1/2$,
\item $Z^{2} \le -M^{2} + 2M - \frac12$.
\end{enumerate}
The target function $f$ is quadratic in $Z$ with maximum at $Z_{0}(M) =\sqrt{M}/2$.
Observe that
\begin{equation*}
  f(Z_{0}(M),M) = -M^{2} + \frac94M - \frac12.
\end{equation*}
This function is quadratic in $M$, with maximum at $M=\frac98>\frac12$.  Therefore, it is maximized by the
largest feasible~$M=1/2$.  In total,
\begin{equation*}
  f(Z,M) \le f(Z_{0}(M),M) \le f(Z_{0}(1/2),1/2) = \frac38
\end{equation*}
for all feasible values of~$(Z,M)$.  Finally, 
$(Z_{0}(1/2),1/2)$ is feasible, since
\begin{equation*}
  Z_{0}(1/2)^{2} = \frac18 < \frac14 = -\frac14 + 1 - \frac12.
\end{equation*}

For the second case, suppose that $M\ge 1/2$.  Then $S=2 M-1$, and we need to optimize
\begin{equation*}
  f(Z,M) = -M^{2} + 2M - \frac12 - Z^{2} + \sqrt{M}\cdot Z\cdot \sqrt{2(1-M)}
\end{equation*}
subject to
\begin{enumerate}
\item $1/2\le M \le 1$,
\item $Z^{2} \le -M^{2} + 2M - \frac12$.
\end{enumerate}

Again, $f$ is quadratic in~$Z$, with maximum at~$Z_{0}(M) = \sqrt{M(1-M)/2}$.  To see that
$(Z_{0}(M),M)$ is feasible, we have to check that
\begin{align*}
  0 \le -M^{2} + 2M - \frac12 - Z_{0}(M)^{2}
  = -\frac12M^{2} + \frac32M - \frac12.
\end{align*}
The right hand side is a quadratic polynomial with zeros at $(3-\sqrt{5})/2<1/2$ and~$(3+\sqrt{5})/2>1$, which proves
that $(M,Z_{0}(M))$ satisfies all constraints.

Therefore, we need to maximize the quadratic function
\begin{equation*}
  f(M) = -M^{2} + 2M - \frac12 + \frac12M(1-M) = -\frac32M^{2} + \frac52M - \frac12
\end{equation*}
with $1/2\le M\le 1$.  The maximum is at $M=5/6$, where the value is
\begin{equation*}
  f(5/6) = - \frac{75}{72} + \frac{25}{12} - \frac12 = \frac{13}{24} = 0.541\overline{6}.\qedhere
\end{equation*}
\end{proof}

\begin{proof}[Proof (of Theorem \ref{mainthm})]
By Lemma \ref{lemmaD} and Lemma \ref{lemmaf},
\begin{align*}
B(G) &= \sum _{i=1}^{n-1} D_i \\
&\le \sum _{i:i \text{ single}} \min(m, i) + \sqrt{m} \sqrt{\sum _{i:i \text{ multiple}}\min(m,i)} \sqrt{\sum _{i:i \text{ multiple}} n_i} \\
&= g(y,z,s,m) \\
&\le \left( 1 + o(1) \right)f(Y,Z,S,M) n^2 \\
&\le \left( 1 + o(1) \right) \frac{13}{24}n^2,
\end{align*}
implying that
\begin{equation*}
\frac{|A(G)|}{\binom{n}{2}} = (1+o(1))\frac{B(G)}{n^2} \le \left( 1 + o(1) \right) \frac{13}{24}.\qedhere
\end{equation*}
\end{proof}
\section*{Acknowledgments}
We acknowledge that it was Aaron Dutle who first gave a non-trivial upper bound $(1-1/e +o(1))n^2=0.63212...n^2$ for the number of non-zero entries in a JDV. His proof was somewhat similar to
Subsection~\ref{sec:4} but missed the symmetry that is key to that subsection. All authors except the second author were supported in part by the U.S.\ Air Force Office of Scientific Research (AFOSR) and the Defense Advanced Research Projects Agency (DARPA). The last author was supported in part by the NSF DMS contracts no.\ 1300547 and 1600811. The authors are also grateful to Sonja Petrovi\'{c} and Despina Stasi for helpful discussions.

\bibliographystyle{apalike}
\bibliography{bib}
\end{document}